\documentclass[a4paper]{amsart}
\usepackage{multicol,ifthen}
\usepackage{amsmath}
\usepackage{amssymb}
\usepackage{amsthm} 
\usepackage{enumerate}

\begin{document}

\newcommand{\F}{\mathcal{F}}
\newcommand{\R}{\mathbb R}
\newcommand{\T}{\mathbb T}
\newcommand{\N}{\mathbb N}
\newcommand{\Z}{\mathbb Z}
\newcommand{\C}{\mathbb C}  
\newcommand{\h}[2]{\mbox{$ \widehat{H^{#1}_{#2}}$}}
\newcommand{\hh}[3]{\mbox{$ \widehat{H^{#1}_{#2, #3}}$}} 
\newcommand{\n}[2]{\mbox{$ \| #1\| _{ #2} $}} 
\newcommand{\x}{\mbox{$X^r_{s,b}$}} 
\newcommand{\xx}{\mbox{$X_{s,b}$}}
\newcommand{\X}[3]{\mbox{$X^{#1}_{#2,#3}$}} 
\newcommand{\XX}[2]{\mbox{$X_{#1,#2}$}}
\newcommand{\q}[2]{\mbox{$ {\| #1 \|}^2_{#2} $}}
\newcommand{\e}{\varepsilon}
\newcommand{\lb}{\langle}
\newcommand{\rb}{\rangle}
\newcommand{\ls}{\lesssim}
\newcommand{\gs}{\gtrsim}
\newcommand{\pd}{\partial}
\newtheorem{lemma}{Lemma} 
\newtheorem{kor}{Corollary} 
\newtheorem{theorem}{Theorem}
\newtheorem{prop}{Proposition}

\title[Space-time estimates for KP-type equations]{Bilinear space-time estimates for linearised KP-type equations on the three-dimensional torus with applications}

\author[A.~Gr{\"u}nrock]{Axel~Gr{\"u}nrock}

\address{Axel~Gr{\"u}nrock: Rheinische Friedrich-Wilhelms-Universit\"at Bonn,
Mathematisches Institut, Beringstra{\ss}e 1, 53115 Bonn, Germany.}
\email{gruenroc@math.uni-bonn.de}

\thanks{The author was partially supported by the Deutsche Forschungsgemeinschaft, Sonderforschungsbereich 611.}

\subjclass{35Q53}

\begin{abstract}
A bilinear estimate in terms of Bourgain spaces associated with a linearised
Kadomtsev-Petviashvili-type equation on the three-dimensional torus is shown. As a
consequence, time localized linear and bilinear space time estimates for this
equation are obtained. Applications to the local and global well-posedness of
dispersion generalised KP-II equations are discussed. Especially it is proved that
the periodic boundary value problem for the original KP-II equation is locally
well-posed for data in the anisotropic Sobolev spaces $H^s_xH^{\e}_y(\T^3)$, if
$s \ge \frac12$ and $\e > 0$.
\end{abstract}

\maketitle

\section{Introduction and main results}

In a recent paper \cite{GPS08} joint with M. Panthee and J. Silva we investigated local
and global well-posedness issues of the Cauchy problem for the dispersion generalised
Kadomtsev-Petviashvili-II (KP-II) equation

\begin{equation}
\label{KPII}
\left\{
\begin{array}{l}
\pd_t u - |D_x|^{\alpha}\pd_x u+\pd_x^{-1}\Delta_{y} u +u\pd_x u=0 \\
u(0,x,y)=u_0(x,y)
\end{array}
\right.
\end{equation}

on the cylinders $\T \times \R$ and $\T \times \R^2$, respectively. We considered data $u_0$ satisfying the mean zero condition

\begin{equation}
\label{meanzero}
\int_0^{2\pi}u_0(x,y) dx =0
\end{equation}

and belonging to the anisotropic Sobolev spaces $H^s_x(\T)H^{\e}_y(\R^{n-1})$, 
$n \in \{2,3\}$. We could prove quite general (with respect to the dispersion
parameter $\alpha$) local well-posedness results, to a large extent optimal - up to the endpoint - 
(with respect to the Sobolev regularity). In two dimensions and for higher
dispersion ($\alpha > 3$) in three dimensions, these local results could be combined
with the conservation of the $L^2$-norm to obtain global well-posedness.

A key tool to obtain these results were certain bilinear space time estimates for
free solutions, similar to Strichartz estimates. A central argument to obtain the space
time estimates was the following simple observation. Consider a linearised version
of \eqref{KPII} with a more general phase function

\begin{equation}
\label{KPlin}
\left\{
\begin{array}{l}
\pd_t u - i\phi(D_x,D_y)u:=\pd_t u -  i\phi_0(D_x)u+\pd_x^{-1}\Delta_{y} u =0 \\
u(0,x,y)=u_0(x,y)
\end{array}
\right.
\end{equation}

where $\phi_0$ is arbitrary at the moment, with solution $u(x,y,t)=e^{it\phi(D_x,D_y)}u_0(x,y)$.
Then we can take the partial Fourier transform $\F_x$ with respect to the first spatial
variable $x$ only to obtain

$$\F_x e^{it\phi(D_x,D_y)}u_0 (k,y) = e^{it\phi_0(k)}e^{i\frac{t}{k}\Delta_y}\F_xu_0(k,y).$$

Fixing $k$ we have a solution of the free Schr\"odinger equation - with rescaled time
variable $s:=\frac{t}{k}$, and multiplied by a phase factor of size one. Now the whole
Schr\"odinger theory - Strichartz estimates, bilinear refinements thereof, local
smoothing and maximal function estimates - is applicable to obtain space time estimates
for the linearised KP-type equation \eqref{KPlin}.

While in two space dimensions this simple argument has to be supplemented by further
estimates depending on $\phi_0$, we could obtain (almost) sharp estimates in the three-dimensional
$\T \times \R^2$-case \emph{only} by using the "Schr\"odinger trick" described
above. In view of Bourgain's $L^4_{xt}$-estimate for free solutions of the Schr\"odinger
equation with data defined on the two-dimensional torus \cite[first part of Prop. 3.6]{B93a}, the question
comes up naturally, if our analysis in \cite{GPS08} concerning $\T \times \R^2$ can
be extended to KP-type equations on $\T^3$, and that's precisely the aim of the present
paper.

To state our main results we have to introduce some more notation: We will consider
functions $u,v,\dots$ of $(x,y,t) \in \T \times \T^2 \times \R $ with Fourier transform
$\widehat{u},\widehat{v},\dots$, sometimes written as $\F u, \F v,\dots$, depending on
the dual variables $(\xi, \tau):=(k,\eta,\tau) \in \Z \times \Z^2 \times \R$. Throughout
the paper we assume $u,v,\dots$ to fulfill the mean zero condition $\widehat{u}(0,\eta,\tau)=0$.
For these functions we define the norms

$$\|u\|_{X_{s,\e,b}}:=\| |k|^{s}\lb\eta\rb^{\e}\lb\sigma\rb^{b}\widehat{u}\|_{L^2_{\xi,\tau}},$$

where $\lb x \rb^2=1+|x|^2$ and $\sigma = \tau-\phi(\xi) = \tau-\phi_0(k)+\frac{|\eta|^2}{k}$. Allthough
some of our arguments do not rely on that, we will always assume $\phi_0$ to be odd, in
order to have $\|u\|_{X_{s,\e,b}}=\|\overline{u}\|_{X_{s,\e,b}}$. For $\e=0$ we
abbreviate $\|u\|_{X_{s,\e,b}}=\|u\|_{X_{s,b}}$. In these terms our central bilinear
space time estimate reads as follows.

\begin{theorem} 
\label{mainest}
Let $b > \frac12$, $s_{1,2} \ge 0$ with $s_1+s_2>1$ and $\e _{0,1,2} \ge 0$ with $\e _0+\e _1+\e _2>0$.
Then the estimate
\begin{equation}\label{bil}
\|D_y^{-\e_0}(uv)\|_{L^2_{xyt}} \ls \|u\|_{X_{s_1,\e _1,b}}\|v\|_{X_{s_2,\e _2,b}}
\end{equation}
and its dualized version
\begin{equation}\label{bil'}
\|uv\|_{X_{-s_1,-\e _1,-b}}\ls \|D_y^{\e_0}u\|_{L^2_{xyt}} \|v\|_{X_{s_2,\e _2,b}}
\end{equation}
hold true.
\end{theorem}

Taking $\e_0=0$ and $u=v$ we obtain the linear estimate

\begin{equation}\label{lin}
\|u\|_{L^4_{xyt}} \ls \|u\|_{X_{s,\e ,b}},
\end{equation}

whenever $s,b > \frac12$ and $\e > 0$. The estimate \eqref{bil} can
be applied to time localised solutions $e^{it\phi(D_x,D_y)}u_0$ and $e^{it\phi(D_x,D_y)}v_0$ of
\eqref{KPlin} to obtain

\begin{equation}
\|D_y^{-\e_0}(e^{it\phi(D_x,D_y)}u_0e^{it\phi(D_x,D_y)}v_0)\|_{L_t^2([0,1],L^2_{xy})}\ls \|u_0\|_{H^{s_1}_xH^{\e _1}_y}\|v_0\|_{H^{s_2}_xH^{\e _2}_y},
\end{equation}

provided $s_{1,2}$ and $\e _{0,1,2}$ fulfill the assumptions in Theorem \ref{mainest}. Especially for $s> \frac12$
and $\e > 0$ we have the linear estimate

\begin{equation}
\|e^{it\phi(D_x,D_y)}u_0\|_{L_t^4([0,1],L^4_{xy})}\ls \|u_0\|_{H^{s}_xH^{\e }_y}.
\end{equation}

The corresponding estimate for data $u_0$ defined on $\R^3$ holds global in time and
has a $\|u_0\|_{H^{\frac12}_xL^{2}_y}$ on the right hand side. It goes back to Ben Artzi and Saut \cite{BAS99}.
Dimensional analysis shows that the Sobolev exponent $s= \frac12$ is necessary. So we
haven't lost more than an $\e$ derivative in the $x$- as well as in the $y$-variable.

In order to prove Theorem \ref{mainest}, we will work in Fourier space, where the product
$uv$ is turned into the convolution

$$\widehat{u}*\widehat{v}(\xi,\tau)=\int d \tau_1 \sum_{\xi_1 \in \Z^3}\widehat{u}(\xi_1,\tau_1)\widehat{v}(\xi_2,\tau_2).$$

Here always $(\xi,\tau)=(k,\eta,\tau)=(k_1+k_2,\eta_1+\eta_2,\tau_1+\tau_2)=(\xi_1+ \xi_2,\tau_1+\tau_2)$.
Observe that there is no contribution to the above sum, whenever $k_1=0$ or $k_2=0$. In
the estimation of such convolutions the $\sigma$-weights in the $X_{s,\e ,b}$-norms
become $\sigma_1 = \tau_1 - \phi(\xi_1)$ and $\sigma_2 = \tau_2 - \phi(\xi_2)$. With
this notation we introduce the bilinear Fourier multiplier $M^{- \e}$, which we define by

$$\F M^{- \e} (u,v)(\xi,\tau):=\chi_{\{k \ne 0\}}\int d \tau_1 \sum_{\xi_1 \in \Z^3}\lb k_1\eta -k \eta_1\rb^{- \e}\widehat{u}(\xi_1,\tau_1)\widehat{v}(\xi_2,\tau_2).$$

Observe that $|k_1\eta -k \eta_1|=|k_1\eta_2 -k_2 \eta_1|$, so that we have symmetry between $u$ and $v$. The 
operator $M^{- \e}$ serves to compensate for the unavoidable loss of the $D_y^{\e}$ in \eqref{bil}.
A careful examination of the proof of Theorem \ref{mainest} will give the following. 

\begin{theorem} 
\label{Meps}
Let $s, b > \frac12$ and $\e > 0$. Then
\begin{equation}\label{Mepsbil}
\|M^{- \e} (u,v)\|_{L^2_{xyt}} \ls \|u\|_{X_{s,b}}\|v\|_{X_{s,b}}.
\end{equation}
\end{theorem}

The proof of the above theorems will be done in section 2, while section 3 is
devoted to the applications. Here we specialize to the dispersion generalised KP-II equation
\eqref{KPII}, that is to $\phi_0(k)=|k|^{\alpha}k$. For $\alpha =2$, which is the original KP-II equation we will use Theorem
\ref{mainest} to show the following local result.

\begin{theorem} 
\label{local}
Let $s \ge \frac12$ and $\e > 0$. Then, for $\alpha =2$, the Cauchy problem \eqref{KPII}
is locally well-posed for data $u_0 \in H^{s}_xH^{\e }_y(\T^3)$ satisfying the mean zero
condition \eqref{meanzero}.
\end{theorem}

For high dispersion, i. e. $\alpha > 3$, one can allow $s<0$ and $\e =0$. In fact, by the
aid of Theorem \ref{Meps} we can prove:

\begin{theorem}
\label{global}
Let $3 < \alpha \le 4$ and $s > \frac{3-\alpha}{2}$. Then the Cauchy problem \eqref{KPII}
is locally well-posed for data $u_0 \in H^{s}_xL^{2 }_y(\T^3)$ satisfying \eqref{meanzero}.
If $s\ge 0$ the corresponding solutions extend globally in time by the conservation of
the $L^{2 }_{xy}$-norm.
\end{theorem}

More precise statements of the last two theorems will be given in section 3. We conclude
this introduction with several remarks commenting on our well-posedness results and their context.

1. Concerning the Cauchy problem for the KP-II equation and its dispersion generalisations
on $\R^2$ and $\R^3$ there is a rich literature, see e. g. \cite{HT}, \cite{H},
\cite{ILM}, \cite{IM01}, \cite{S93}, \cite{ST99}, \cite{T&T01},\cite{T99}, this list ist by no means
exhaustive. For $\alpha = 2$ the theory has even been pushed to the critical space in
a recent work of Hadac, Herr, and Koch \cite{HHK}. On the other hand, for the periodic or semiperiodic
problem the theory is much less developed. Besides Bourgain's seminal paper \cite{B93}
our only references here are the papers \cite{ST00}, \cite{ST01} of Saut and Tzvetkov
and our own contribution \cite{GPS08} joint with M. Panthee and J. Silva.

2. The results obtained here for the fully periodic case are as good as those in
\cite{GPS08} for the $\T \times \R^2$ case and even as those obtained by Hadac
\cite{HT} for $\R^3$, which are optimal by scaling considerations. We believe this is
remarkable since apart from nonlinear wave and Klein-Gordon equations there are
only very few examples in the literature, where the periodic problem is as well behaved
as the corresponding continuous case. (One example is of course Bourgain's $L^2_x(\T)$
result for the cubic Schr\"odinger equation \cite{B93a}, but this is half a derivative
away from the scaling limit.) On the other hand there are many examples, such as KdV
and mKdV, where at least the methods applied here lead to (by $\frac14$ derivative)
weaker results for the periodic problem. Another example is the KP-II equation itself in
two space dimensions, where in \cite{GPS08} we lost $\frac14$ derivative when stepping
from $\R^2$ to $\T \times \R$. Another loss of $\frac14$ derivative in the step from
$\T \times \R$ to $\T^2$ is probable.

3. For the semilinear Schr\"odinger equation

$$iu_t + \Delta u = |u|^pu$$

on the torus, with $2<p<4$ in one, $1<p<2$ in two dimensions, one barely misses
the conserved $L_x^2$ norm and thus cannot infer global well-posedness. The reason
behind that is the loss of an $\e$ derivative in the Strichartz type estimates in
the periodic case. A corresponding derivative loss is apparent in Theorem \ref{mainest}
but the usually ignored mixed part of the rather comfortable resonance relation of the
dispersion generalised KP-II equation allows (via $M^{-\e}$) to compensate for this
loss, so that for high dispersion ($3 < \alpha \le 4$) we can obtain something global.
The author did not expect that, when starting this investigation.

4. We restrict ourselves to the most important (as we believe) values of $\alpha$.
Our arguments work as well for $\alpha \in (2,3]$ with optimal lower bound for $s$
but possibly with an $\e$ loss in the $y$ variable. For $\alpha > 4 $ we probably
loose optimality.

5. In \cite{T&T} Takaoka and Tzvetkov proved a time localised $L^4-L^2$ Strichartz
type estimate \emph{without} derivative loss for free solutions of the Schr\"odinger
equation with data defined on $\R \times \T$. Inserting their arguments in our proof
of Theorem \ref{mainest} we can show a variant thereof with $\e_0=\e_1=\e_2=0$, if
the data live on $\T \times \R \times \T$. Consequently our well-posedness results
are valid in this case, too.

\section{Proof of Theorem \ref{mainest}}

The main ingredient in the proof of Bourgain's Schr\"odinger estimate

$$\|e^{it \Delta}u_0\|_{L^4_{xt}(\T ^3)} \ls \|u_0\|_{H^{\e}_x(\T ^2)}, \quad (\e > 0)$$

is the well known estimate on the number of representations of an integer $r>0$ as a
sum of two squares: For any $\e > 0$ there exists $c_{\e}$ such that

\begin{equation}\label{338}
\# \{\eta \in \Z ^2 : |\eta|^2=r\} \le c_{\e}r^{\e}.
\end{equation}

For \eqref{338}, see \cite[Theorem 338]{HW}. Our proof of Theorem \ref{mainest} relies
on the following variant thereof.

\begin{lemma}\label{diophant}
Let $r \in \N$, $\delta \in \R^2$. Then for any $\e > 0$ there exists $c_{\e}$, independent of $r$ and $\delta$, such that
\begin{equation}\label{2.3}
\# \{\eta \in \Z ^2 : r \le |\eta - \delta|^2 <r+1\} \le c_{\e}r^{\e}.
\end{equation}
\end{lemma}

\begin{proof}
In the case where $\delta \in \Z ^2$, this follows by translation from \eqref{338}.
So we may assume $\delta \in [0,1]^2$, and we start by considering the special case
$\delta = (\frac12, \frac12)$. Here 

\begin{eqnarray*}
 & \# \{\eta \in \Z ^2 : r \le |\eta - \delta|^2 <r+1\}  \\
 = & \# \{\eta \in \Z ^2 : 4 r \le |2\eta - 2\delta|^2 < 4(r+1)\} \\
 = & \sum_{l=4r}^{4r+3} \# \{\eta \in \Z ^2 :   |2\eta - 2\delta|^2 =l\}.
\end{eqnarray*}

But $|2\eta - 2\delta|^2 = 4|\eta|^2-4\lb \eta , 2 \delta \rb + 2 \equiv 2 \quad (\rm{mod}\, 4)$,
so the only contribution to the above sum comes from $l=4r+2$. Thus, by \eqref{338}, for
any $\e ' >0$ there exists $c_{\e '}$ such that

\begin{equation}\label{2.4}
\# \{\eta \in \Z ^2 : r \le |\eta - \delta|^2 <r+1\}
\le c_{\e '}(4r+2)^{\e '} \le c_{\e '}(6r)^{\e '}.
\end{equation}

Next we observe that for $\delta \in \{(0, \frac12),(\frac12, 0),(\frac12, 1),(1, \frac12)\}$
we have $|2\eta - 2\delta|^2 \equiv 1 \quad (\rm{mod} \, 4)$, so that the estimate \eqref{2.4} is valid
in these cases, too. Iterating the argument, we obtain for $\delta = (\frac{m_1}{2^m},\frac{m_2}{2^m})$
with $m \in \N$ and $0 \le m_{1,2} \le 2^m$ the estimate

\begin{equation}\label{2.5}
\# \{\eta \in \Z ^2 : r \le |\eta - \delta|^2 <r+1\}\le c_{\e '}(6^mr)^{\e '}.
\end{equation}

Now for an arbitrary $\delta \in [0,1]^2$ we choose $\delta ' = (\frac{m_1}{2^{m'}},\frac{m_2}{2^{m'}})$
with $|\delta  - \delta '| \sim r^{- \frac12}$, so that

$$\{\eta \in \Z ^2 : r \le |\eta - \delta|^2 <r+1\} \subset \{\eta \in \Z ^2 : r-1 \le |\eta - \delta|^2 <r+2\}$$

and hence, by \eqref{2.5},

$$\# \{\eta \in \Z ^2 : r \le |\eta - \delta|^2 <r+1\}\le 3c_{\e '}(6^{m'}r)^{\e '}.$$

Such a $\delta $ exists for $2^{m'} \sim r^{\frac12}$, estimating roughly, for $6^{m'} \le r^{\frac32}$.
So we have the bound

$$\# \{\eta \in \Z ^2 : r \le |\eta - \delta|^2 <r+1\}\le 3c_{\e '}r^{\frac{5\e '}{2}}.$$

Choosing $\e ' = \frac{2 \e}{5}$, $c_{\e }= 3 c_{\e '}$, we obtain \eqref{2.3}.

\end{proof}

\begin{kor}
\label{elephant}
If $B$ is a disc (or square) of arbitrary position and of radius (sidelength) $R$,
then for any $\e > 0$ there exists $c_{\e}$ such that
\begin{equation}\label{2.6}
\sum_{\substack{\eta_1 \in \Z ^2 \\ r \le |\eta_1 - \delta|^2 <r+1}}\chi_B (\eta_1) \le c_{\e }R^{\e }.
\end{equation}
\end{kor}

\begin{proof}
If $R \gs r^{\frac16}$, the estimate \eqref{2.6} follows from Lemma \ref{diophant}.
If $R \ll r^{\frac16}$, there are at most two lattice points on the intersection of $B$
with the circle of radius $ \simeq r^{\frac12}$ around $\delta$, by Lemma 4.4 of \cite{dSPST}.
\end{proof}

In the sequel we will use the following projections: For a subset $M \subset \Z^2$ we define
$P_M$ by $\F P_Mu(k,\eta,\tau)= \chi_M(\eta)\F u(k,\eta,\tau)$. Especially, if $M$ is a
ball of radius $2^l$ centered at the origin, we will write $P_l$ instead of $P_M$. Furthermore
we have $P_{\Delta l}=P_l - P_{l-1}$, and the $P$-notation will also be used in connection
with a sequence $\{Q_{\alpha}^l\}_{\alpha \in \Z^2}$ of squares of sidelength $2^l$, centered
at $2^l\alpha$. Double sized squares with the same centers will be denoted by $\widetilde{Q}_{\alpha}^l$.

\begin{theorem}
\label{central}
Let $s>1$, $b > \frac12$ and $\e >0$. Then for a disc (or square) $B$ of arbitrary position with
radius (sidelength) $R$ we have
\begin{equation}\label{2.7}
\| (P_Bu)v\|_{L^2_{xyt}} \ls R^{\e} \|u\|_{X_{0,b}}\|v\|_{X_{s,b}}.
\end{equation}
\end{theorem}

\begin{proof}
Choose $f,g$ with $\|f\|_{L^2_{\xi \tau}}=\|u\|_{X_{0,b}}$ and $\|g\|_{L^2_{\xi \tau}}=\|v\|_{X_{s,b}}$.
Then the left hand side of \eqref{2.7} becomes
\begin{equation}\label{2.8}
\|\int d\tau_1 \sum_{k_1\in\Z}\sum_{\eta_1 \in \Z^2}\chi_B (\eta_1)f(\xi_1,\tau_1)
\lb \sigma_1\rb^{-b}g(\xi_2,\tau_2)|k_2|^{-s}\lb \sigma_2\rb^{-b}\|_{L^2_{\xi \tau}}.
\end{equation}
Since $\| uv\|_{L^2_{xyt}}=\| u\overline{v}\|_{L^2_{xyt}}$, which corresponds to
$\|\widehat{u}*\widehat{v}\|_{L^2_{\xi \tau}}=\|\widehat{u}*\widehat{w}\|_{L^2_{\xi \tau}}$
on Fourier side, where $\widehat{w}(\xi, \tau)= \overline{\widehat{v}}(-\xi, -\tau)$,
and since the phase function $\phi$ is assumed to be odd, so that $\|u\|_{X_{s,b}}=\|\overline{u}\|_{X_{s,b}}$,
we may assume in the estimation on \eqref{2.8}, that $k_1$ and $k_2$ have the same sign,
cf. Remark 4.7 in \cite{B93}. So it's sufficient to consider $0 < |k_2| \le |k_1| < |k|$. Now,
using Minkowski's inequality we estimate \eqref{2.8} by

\begin{eqnarray*}
\|\sum_{k_1\in\Z}|k_2|^{-s}\|\int d\tau_1\sum_{\eta_1 \in \Z^2}\chi_B (\eta_1)
f(\xi_1,\tau_1)\lb \sigma_1\rb^{-b}g(\xi_2,\tau_2)\lb \sigma_2\rb^{-b}\|_{L^2_{\eta \tau}}\|_{L^2_k}\\
\le \||k_2|^{-\frac12}\|\int d\tau_1\sum_{\eta_1 \in \Z^2}\chi_B (\eta_1)
f(\xi_1,\tau_1)\lb \sigma_1\rb^{-b}g(\xi_2,\tau_2)\lb \sigma_2\rb^{-b}\|_{L^2_{\eta \tau}}\|_{L^2_{kk_1}},
\end{eqnarray*}
where Cauchy-Schwarz was applied to $\sum_{k_1\in\Z}$. Thus it is sufficient to show that
\begin{eqnarray}
\label{2.9}
\|\int d\tau_1\sum_{\eta_1 \in \Z^2}\chi_B (\eta_1)
f(\xi_1,\tau_1)\lb \sigma_1\rb^{-b}g(\xi_2,\tau_2)
\lb \sigma_2\rb^{-b}\|_{L^2_{\eta \tau}} \\
\ls R^{\e}|k_2|^{\frac12}\|f(k_1, \cdot,\cdot)\|_{L^2_{\eta_1 \tau_1}}
\|g(k_2, \cdot,\cdot)\|_{L^2_{\eta \tau}}. \nonumber
\end{eqnarray}
By the "Schwarz-method" developed in \cite{KPV96a}, \cite{KPV96b} and by \cite[Lemma 4.2]{GTV97},
\eqref{2.9} follows from
\begin{equation}\label{2.10}
\sum_{\eta_1 \in \Z^2}\chi_B (\eta_1) \lb \tau -\phi_0 (k_1)-\phi_0 (k_2)+ 
\frac{|\eta_1|^2}{k_1}+ \frac{|\eta_2|^2}{k_2}\rb^{-2b} \ls R^{2\e}|k_2|.
\end{equation}
For $\omega:=\eta_1-\frac{k_1}{k}\eta$ we have $\frac{|\eta_1|^2}{k_1}+ \frac{|\eta_2|^2}{k_2}=\frac{|\eta|^2}{k}+ \frac{k}{k_1k_2}|\omega|^2$,
so that with $a:=\tau -\phi_0 (k_1)-\phi_0 (k_2)+\frac{|\eta|^2}{k}$ the left hand side of \eqref{2.10} becomes

$$\sum_{\eta_1 \in \Z^2}\chi_B (\eta_1)\lb a + \frac{k}{k_1k_2}|\omega|^2\rb^{-2b} = \sum_{r \ge 0}\lb a + \frac{k}{k_1k_2}r\rb^{-2b} \sum_{r\le |\eta_1 - \frac{k_1}{k}\eta|^2<r+1}\chi_B (\eta_1).$$

By Corollary \ref{elephant} the inner sum is controlled by $c_{\e}R^{\e}$, while

$$\sum_{r \ge 0}\lb a + \frac{k}{k_1k_2}r\rb^{-2b} \ls \frac{|k_1k_2|}{|k|}\ls |k_2|,$$

which proves \eqref{2.10}.
\end{proof}

\emph{Remark:}

The quantity, which we precisely loose in the application of Lemma \ref{diophant}, is

$$r^{\e}\simeq |\eta_1 - \frac{k_1}{k}\eta|^{2 \e} \le \lb k\eta_1 - k_1 \eta\rb^{2 \e},$$

which is the symbol of the Fourier multiplier $M^{2 \e}$. Rereading carefully the calculation in the previous proof,
we see that - instead of \eqref{2.9} - the following estimate holds true as well.

\begin{eqnarray}
\label{2.12}
\|\int d\tau_1\sum_{\eta_1 \in \Z^2}\lb k\eta_1 - k_1 \eta\rb^{- \e}
f(\xi_1,\tau_1)\lb \sigma_1\rb^{-b}g(\xi_2,\tau_2)
\lb \sigma_2\rb^{-b}\|_{L^2_{\eta \tau}} \\
\ls \frac{|k_1k_2|^{\frac12}}{|k|^{\frac12}}\|f(k_1, \cdot,\cdot)\|_{L^2_{\eta_1 \tau_1}}
\|g(k_2, \cdot,\cdot)\|_{L^2_{\eta \tau}}. \nonumber
\end{eqnarray}

(Introducing the $M^{- \e}$ we cannot justify the sign assumption on $k, k_{1,2}$ any more.)
Multiplying by $|k|^{\frac12}$ and summing up over $k_1$ using Cauchy-Schwarz we obtain

\begin{equation}\label{2.13}
\|\F _x D_x^{\frac12}M^{- \e}(u,v)\|_{L^{\infty}_k L^2_{yt}}\ls \|u\|_{X_{\frac12,b}}\|v\|_{X_{\frac12,b}},
\end{equation}

from which \eqref{Mepsbil} follows by a further application of the Cauchy-Schwarz inequality. So Theorem \ref{Meps} is proved.

\begin{proof}[Proof of Theorem \ref{mainest}]
Since in Corollary \ref{elephant} the position of the disc is arbitrary, we may
replace $\chi_B (\eta_1)$ by $\chi_B (\eta_2)$ in the proof of Theorem \ref{central}, which gives

\begin{equation}\label{2.11}
\|u (P_Bv)\|_{L^2_{xyt}} \ls R^{\e} \|u\|_{X_{0,b}}\|v\|_{X_{s,b}}.
\end{equation}

Now we have symmetry between $u$ and $v$, so that we may interpolate bilinearly
to obtain

\begin{equation}\label{2.20}
\| (P_Bu)v\|_{L^2_{xyt}} \ls R^{\e} \|u\|_{X_{s_1,b}}\|v\|_{X_{s_2,b}}
\end{equation}

for $s_{1,2} \ge 0$ with $s_1+s_2>1$. Decomposing dyadically we obtain with $0 < \e' < \e$

\begin{eqnarray*}
\|uv\|_{L^2_{xyt}} \le & \sum_{l \ge 0} \|(P_{\Delta l}u)v\|_{L^2_{xyt}} 
\ls  \sum_{l \ge 0} 2^{l \e'} \|P_{\Delta l}u\|_{X_{s_1 ,b}}\|v\|_{X_{s_2,b}} \\
\ls & \sum_{l \ge 0} 2^{l (\e' - \e)}\|u\|_{X_{s_1,\e ,b}}\|v\|_{X_{s_2,b}} \ls \|u\|_{X_{s_1,\e ,b}}\|v\|_{X_{s_2,b}}.
\end{eqnarray*}

Exchanging $u$ and $v$ again we have shown for $s_{1,2} \ge 0$ with $s_1+s_2>1$
and $\e_{1,2} \ge 0$ with $\e_1+ \e_2 >0$ that

\begin{equation}\label{2.21}
\| uv\|_{L^2_{xyt}} \ls  \|u\|_{X_{s_1,\e_1,b}}\|v\|_{X_{s_2, \e_2,b}},
\end{equation}

which is the $\e_0=0$ part of \eqref{bil} in Theorem \ref{mainest}. To see the $\e_0>0$ part, we decompose

$$\|D_y^{-\e_0}(uv)\|_{L^2_{xyt}} \le \sum_{l \ge 0} 2^{-l \e_0} \|P_{\Delta l}(uv)\|_{L^2_{xyt}},$$

where for fixed $l$

$$\|P_{\Delta l}(uv)\|^2_{L^2_{xyt}}=  \sum_{\alpha, \beta \in \Z^2}\lb P_{\Delta l}((P_{Q^l_{\alpha}}u)v), P_{\Delta l}((P_{Q^l_{\beta}}u)v)\rb_{L^2_{xyt}}.$$

Now for $\eta_1 \in Q^l_{\alpha}$, $|\eta| \le 2^l$ we have $\eta_2=\eta-\eta_1 \in \widetilde{Q}^l_{-\alpha}$, so that
the latter can be estimated by

\begin{multline*}
\hspace{1cm} \sum_{\alpha, \beta \in \Z^2}\lb (P_{Q^l_{\alpha}}u)(P_{\widetilde{Q}^l_{-\alpha}}v), (P_{Q^l_{\beta}}u)(P_{\widetilde{Q}^l_{-\beta}}v)\rb_{L^2_{xyt}}  \\
\hspace{1cm}\le  \sum_{\alpha, \beta \in \Z^2} \lb (P_{\widetilde{Q}^l_{\alpha}}u)(P_{\widetilde{Q}^l_{\beta}}\overline{v}),(P_{\widetilde{Q}^l_{\beta}}u)(P_{\widetilde{Q}^l_{\alpha}}\overline{v})\rb_{L^2_{xyt}}\hfill\\
\hspace{1cm}\le  \sum_{\alpha, \beta \in \Z^2} \|(P_{\widetilde{Q}^l_{\alpha}}u)(P_{\widetilde{Q}^l_{-\beta}}v)\|_{L^2_{xyt}}\|(P_{\widetilde{Q}^l_{\beta}}u)(P_{\widetilde{Q}^l_{-\alpha}}v)\|_{L^2_{xyt}} \hfill\\
\hspace{1cm}\le  \sum_{\alpha, \beta \in \Z^2} \|(P_{\widetilde{Q}^l_{\alpha}}u)(P_{\widetilde{Q}^l_{-\beta}}v)\|^2_{L^2_{xyt}}.\hfill
\end{multline*}

Using \eqref{2.20} and the almost orthogonality of the sequence $\{P_{\widetilde{Q}^l_{\alpha}}v\}_{\alpha \in \Z^2}$ we estimate
the latter by

$$2^{2l\e}\sum_{\alpha, \beta \in \Z^2}\|P_{\widetilde{Q}^l_{\alpha}}u\|^2_{X_{s_1,b}}\|P_{\widetilde{Q}^l_{-\beta}}v\|^2_{X_{s_2,b}}\ls 2^{2l\e}\|u\|^2_{X_{s_1,b}}\|v\|^2_{X_{s_2,b}}.$$

Choosing $\e < \e_0$ the sum over $l$ remains finite and we arrive at

$$\|D_y^{-\e_0}(uv)\|_{L^2_{xyt}} \ls \|u\|_{X_{s_1,b}}\|v\|_{X_{s_2,b}}.$$

Finally we remark that \eqref{bil} and \eqref{bil'} are equivalent by duality.
\end{proof}

\section{Applications to KP-II type equations}

Here the phase function is specified as $\phi_0 (k) = |k|^{\alpha}k$, $\alpha \ge 2$,
so that the mixed weight becomes $\sigma = \tau - |k|^{\alpha}k + \frac{|\eta|^2}{k}$.
To prove the well-posedness results in Theorem \ref{local} and \ref{global}, we need
some more norms and function spaces, respectively. In both cases we use the spaces
$X_{s,\e,b; \beta}$ with additional weights, introduced in \cite{B93} and defined by 

\begin{equation}\label{beta}
\|f\|_{X_{s,\e,b;\beta}} := \left\|\langle k\rangle^{s}\lb\eta\rb^{\e} \langle \sigma \rangle^b
\Big(1+\frac{\langle \sigma \rangle}{\langle k\rangle^{\alpha+1}}\Big)^{\beta}\hat f\right\|_{L^2_{\tau \xi}},
\end{equation}

We will always have $\beta \ge 0$, so that 

\begin{equation}\label{5.1}
\|f\|_{X_{s,b}}\leq \|f\|_{X_{s,b;\beta}}.
\end{equation}

Observe that

\begin{equation}\label{5.2}
\|f\|_{X_{s,b}}\sim \|f\|_{X_{s,b;\beta}},
\end{equation}
if $\lb \sigma \rb \le \langle k\rangle^{\alpha+1}$.

The case $\alpha = 2$ corresponding to the original KP-II equation becomes a limiting case
in our considerations, where we have to choose the parameter $b= \frac12$. Thus we also
need the auxiliary norms

\begin{equation}\label{ys}
\|f\|_{Y_{s,\e;\beta}}:=\left\|\langle k\rangle^{s}\lb\eta\rb^{\e} \langle \sigma \rangle^{-1}
\Big(1+\frac{\langle \sigma \rangle}{\langle k\rangle^{\alpha+1}}\Big)^{\beta}\hat f\right\|_{L^2_{\xi}(L^1_{\tau })},
\end{equation}
cf. \cite{GTV97}, and
\begin{equation}\label{zs}
\|f\|_{Z_{s,\e;\beta}}:=\|f\|_{Y_{s,\e;\beta}}+\|f\|_{X_{s,\e, -\frac12 ;\beta}}.
\end{equation}

As before, for $\e =0$ we will write $X_{s,b;\beta}$ instead of $X_{s,\e,b;\beta}$,
and if the exponent $\beta$ of the additional weight is zero, we use $X_{s,\e,b}$ as
abbreviation for $X_{s,\e,b;\beta}$. Similar for the $Y$- and $Z$-norms.
In these terms the crucial bilinear estimate leading to Theorem \ref{local} is the following.

\begin{lemma}\label{est0}
Let $\alpha =2$, $s \ge \frac12$ and $\e > 0$. Then there exists $\gamma > 0$, 
such that for all $u,v$ supported in $[-T,T] \times \T ^3$ the estimate
\begin{equation}\label{bexy3}
\|\pd_x(uv)\|_{Z_{s,\e;\frac12}} \lesssim  T ^{\gamma}\|u\|_{X_{s,\e, \frac12 ; \frac12}} \|v\|_{X_{s,\e, \frac12 ; \frac12}}
\end{equation}
holds true.
\end{lemma}

Correspondingly for Theorem \ref{global} we have 

\begin{lemma}\label{est1}
Let $3<\alpha \leq 4$. Then, for $s > \frac{3-\alpha}{2}$ there exist $b'>-\frac{1}{2}$
and $\beta \in [0,-b']$, such that for all $b>\frac{1}{2}$
\begin{equation}\label{nonlin1}
\|D^{s+1+\e}_x M^{-\e}(u,v)\|_{X_{0,b';\beta}}\ls \|u\|_{X_{s,b;\beta}}\|v\|_{X_{s,b;\beta}},
\end{equation}
whenever $\e > 0$ is sufficiently small, and
\begin{equation}\label{nonlin2}
\|\pd_x(uv)\|_{X_{s,b';\beta}}\ls \|u\|_{X_{s,b;\beta}}\|v\|_{X_{s,b;\beta}}.
\end{equation}
\end{lemma}

In the proof of both Lemmas above the resonance relation for the KP-II-type
equation with quadratic nonlinearity plays an important role. We have
\begin{equation}\label{rr1}
 \sigma_1+\sigma_2-\sigma = r(k,k_1) + \frac{|k\eta_1-k_1\eta|^2}{kk_1k_2},
\end{equation}

where
 
$$|r(k, k_1)| =| |k|^{\alpha}k-|k_1|^{\alpha}k_1-|k_2|^{\alpha}k_2 |\sim |k_{max}|^{\alpha}|k_{min}|,$$

see \cite{H}. Both terms on the right of \eqref{rr1} have the same sign, so that

\begin{equation}\label{rr}
 \max\{|\sigma|, |\sigma_1|, |\sigma_2|\} \gs |k_{min}| |k_{max}|^{\alpha} +
 \frac{|k\eta_1-k_1\eta|^2}{|kk_1k_2|}.
\end{equation}

The proof of Lemma \ref{est0} is almost the same as that of  Lemma 4 in \cite{GPS08}, it is repeated here - with minor modifications -
for the sake of completeness. We need a variant of Theorem \ref{mainest} with $b < \frac12$. To obtain this, we first
observe that, if $s_{1,2}\ge 0$ with $s_1+s_2> \frac12$, $\e_{0,1,2} \ge 0$ with $\e_0 +\e_1+\e_2 >1$,
$1 \le p \le 2$, and $b > \frac1{2p}$, then
\begin{equation}\label{x}
\|\F D_y^{-\e_0}(uv)\|_{L^2_{\xi}L^p_{\tau}} \ls \|u\|_{X_{s_1,\e_1,b}}\|v\|_{X_{s_2,\e_2,b}}.
\end{equation}
This follows from Sobolev type embeddings and applications of Young's inequality. 
Dualizing the $p=2$ part of \eqref{x} we obtain
\begin{equation}\label{x'}
\|uv\|_{X_{-s_1,-\e_1,-b}}\ls \| D_y^{\e_0}u\|_{L^2_{xyt}}\|v\|_{X_{s_2,\e_2,b}}.
\end{equation}
Now bilinear interpolation with Theorem \ref{mainest} gives the following.

\begin{kor}
\label{varstr}
Let $s_{1,2}\ge 0$ with $s_1+s_2=1$ and $\e_{1,2} \ge 0$ with $\e_1+\e_2 >0$, then there
exist $b < \frac12$ and $p<2$ such that
\begin{equation}\label{xx}
\|\F (uv)\|_{L^2_{\xi}L^p_{\tau}}+\| uv\|_{L^2_{xyt}} \ls \|u\|_{X_{s_1,\e_1,b}}\|v\|_{X_{s_2,\e_2,b}}
\end{equation}
and 
\begin{equation}\label{xx'}
\|uv\|_{X_{-s_1,-b}} \ls \| D_y^{\e _1}u\|_{L^2_{xyt}} \|v\|_{X_{s_2,\e_2,b}}
\end{equation}
hold true.
\end{kor}

The purpose of the $p<2$ part in the above Corollary is to deal with the $Y$-contribution to the $Z$-norm
in Lemma \ref{est0}. Its application will usually follow on an embedding
$$\|\lb \sigma \rb^{-\frac12}\widehat{f}\|_{L^2_{\xi}L^1_{\tau}} \ls \|\widehat{f}\|_{L^2_{\xi}L^p_{\tau}} ,$$
where $p<2$ but arbitrarily closed to $2$. Now we're prepared to establish Lemma \ref{est0}.

\begin{proof}[Proof of Lemma \ref{est0}] 
Without loss of generality we may assume that $s = \frac12$. The proof consists of the following case
by case discussion.

{\bf Case a:} $\lb k\rb^3 \le \lb \sigma \rb$. First we observe that
\begin{equation}\label{+++}
\|\pd_x(uv)\|_{Z_{s,\e;\frac12}} \lesssim \|D_x^{s+1}(D_y^{\e}u\cdot v)\|_{Z_{0,0;\frac12}}+\|D_x^{s+1}(u\cdot D_y^{\e}v)\|_{Z_{0,0;\frac12}}.
\end{equation}
The first contribution to \eqref{+++} equals
\begin{equation*}
\begin{split}
\|\F(D_y^{\e}u\cdot v)\|_{L^2_{\xi,\tau}} + \|\lb \sigma \rb^{-\frac12}\F(D_y^{\e}u\cdot v)\|_{L^2_{\xi}L^1_{\tau}} \\
\ls \|\F(D_y^{\e}u\cdot v)\|_{L^2_{\xi,\tau} \cap L^2_{\xi}L^p_{\tau}} \ls \|u\|_{X_{s,\e,b}}\|v\|_{X_{s,\e,b}}
\end{split}
\end{equation*}
by \eqref{xx}, for some $b < \frac12$. Using the fact\footnote{for a proof see e. g. Lemma 1.10 in \cite{AG}} that
under the support assumption on $u$ the inequality
\begin{equation}\label{time}
\|u\|_{X_{s,\e,b}}\ls T^{\tilde{b}-b}\|u\|_{X_{s,\e,\tilde{b}}}
\end{equation}
holds, whenever $-\frac12 < b < \tilde{b} < \frac12$, this can, for some $\gamma > 0$, be further estimated by
$T^{\gamma}\|u\|_{X_{s,\e, \frac12 ; \frac12}} \|v\|_{X_{s,\e, \frac12 ; \frac12}}
$ 
as desired. The second contribution to \eqref{+++} can be treated in presicely the same manner.

{\bf Case b:} $\lb k\rb^3 \ge \lb \sigma \rb$. Here the additional weight on the left is of size one, so
that we have to show
$$\|\pd_x(uv)\|_{Z_{s,\e}} \lesssim T^{\gamma}\|u\|_{X_{s,\e, \frac12 ; \frac12}} \|v\|_{X_{s,\e, \frac12 ; \frac12}}
. $$

{\bf Subcase b.a:} $\sigma$ maximal. Exploiting the resonance relation \eqref{rr}, we see that the contribution
from this subcase is bounded by
$$\|\F D_xD_y^{\e}(D_x^{-\frac12}u\cdot D_x^{-\frac12}v)\|_{L^2_{\xi,\tau} \cap L^2_{\xi}L^p_{\tau}} \ls \|\F(D_x^{\frac12}D_y^{\e}u\cdot D_x^{-\frac12}v)\|_{L^2_{\xi,\tau} \cap L^2_{\xi}L^p_{\tau}} + \dots ,$$
where $p<2$. The dots stand for the other possible distributions of derivatives on the two factors,
in the same norms, which - by \eqref{xx} of Corollary \ref{varstr} - can all be estimated by $c\|u\|_{X_{s,\e,b}}\|v\|_{X_{s,\e,b}}$
for some $b < \frac12$. The latter is then further treated as in case a.

{\bf Subcase b.b:} $\sigma_1$ maximal. Here we start with the observation that by Cauchy-Schwarz and \eqref{time},
for every $b' > -\frac12$ there is a $\gamma > 0$ such that
$$\|\pd_x(uv)\|_{Z_{s,\e}} \lesssim T^{\gamma}\|D_x^{s+1}(uv)\|_{X_{0,\e,b'}}.$$
With the notation $\Lambda^b = \F ^{-1}\lb \sigma \rb^b \F$ we obtain from the resonance relation that
\begin{equation*}
\begin{split}
&\|D_x^{s+1}(uv)\|_{X_{0,\e,b'}} \ls \|D_x(D_x^{-\frac12}\Lambda^{\frac12}u \cdot D_x^{-\frac12}v)\|_{X_{0,\e,b'}}\\
\ls & \|(D_x^{\frac12}D_y^{\e}\Lambda^{\frac12}u)  (D_x^{-\frac12}v)\|_{X_{0,b'}}  + \|(D_x^{\frac12}\Lambda^{\frac12}u)  (D_x^{-\frac12}D_y^{\e}v)\|_{X_{0,b'}}\\
+ & \|(D_x^{-\frac12}D_y^{\e}\Lambda^{\frac12}u)  (D_x^{\frac12}v)\|_{X_{0,b'}}  + \|(D_x^{-\frac12}\Lambda^{\frac12}u)  (D_x^{\frac12}D_y^{\e}v)\|_{X_{0,b'}}
\end{split}
\end{equation*}
Using \eqref{xx'} the first two contributions can be estimated by $c\|u\|_{X_{s,\e,\frac12}}\|v\|_{X_{s,\e,b}}$
as desired. The third and fourth term only appear in the frequency range $|k| \ll |k_1| \sim |k_2|$, where
the additional weight in the $\|u\|_{X_{s,\e,\frac12 ; \frac12}}$-norm on the right becomes $\frac{|k_2|}{|k_1|}$,
thus shifting a whole derivative from the high frequency factor $v$ to the low frequency factor $u$. So,
using \eqref{xx'} again, these contributions can be estimated by 
$$c\|u\|_{X_{s,\e,\frac12 ; \frac12}}\|v\|_{X_{s,\e,b}}\ls \|u\|_{X_{s,\e,\frac12 ; \frac12}}\|v\|_{X_{s,\e,b ; \frac12}}.$$

\end{proof}

Now we turn to the proof of Lemma \ref{est1}, where the restrictions to the $b$-parameters
can be relaxed slightly, so that the auxiliary $Y$- and $Z$-norms are not needed.
We use again the $\Lambda$-notation, i. e. $\Lambda^b = \F ^{-1}\lb \sigma \rb^b \F$.

\begin{proof}[Proof of Lemma \ref{est1}] 

First we show how \eqref{nonlin1} implies \eqref{nonlin2}. By the resonance relation
\eqref{rr} we have

$$|k_1\eta-k \eta_1|^2 \ls |kk_1k_2|(\lb \sigma \rb+\lb \sigma_1 \rb+\lb \sigma_2 \rb)\le|kk_1k_2|\lb \sigma \rb\lb \sigma_1 \rb\lb \sigma_2 \rb,$$

so that \eqref{nonlin2} is reduced to

$$\|D^{s+1+\frac{\e}{2}}_x M^{-\e}(u,v)\|_{X_{0,b'+\frac{\e}{2};\beta}}\ls \|u\|_{X_{s-\frac{\e}{2},b-\frac{\e}{2};\beta}}\|v\|_{X_{s-\frac{\e}{2},b-\frac{\e}{2};\beta}},
$$

Relabelling appropriately and choosing $\e$ sufficiently small, we see that \eqref{nonlin2} follows
from \eqref{nonlin1}. To prove the latter, we may assume $s\le 0$. Next we choose $\e$
small and $b'$ close to $-\frac12$ so that

\begin{equation}\label{s}
s > 2 + (\alpha +1)b' + 3 \e
\end{equation}

and $\beta:=\frac{s-b'}{\alpha}\in[0,-b']$. Now the proof consists again of a case
by case discussion.

{\bf Case a:} $\lb k\rb^{\alpha +1} \le \lb \sigma \rb$. Here it is sufficient to show

\begin{equation}\label{pf3}
\|D^{s+1+\e -\alpha \beta -\beta}_x M^{-\e}(u,v)\|_{X_{0,b'+\beta}}\ls \|u\|_{X_{s,b}}\|v\|_{X_{s,b}}
\end{equation}

{\bf Subcase a.a:} $|k| \ll |k_1| \sim |k_2|$.

{\bf Subsubcase triple a:} $\lb \sigma \rb \ge \lb \sigma_{1,2} \rb$. Here we use the resonance
relation \eqref{rr} to see that the left hand side of \eqref{pf3} is bounded by

\begin{eqnarray*}
& \|D^{s+1+\e -\alpha \beta +b'}_x M^{-\e}(D_x^{\frac{\alpha(b'+\beta)}{2}}u,D_x^{\frac{\alpha(b'+\beta)}{2}}v)\|_{L^2_{xyt}} \\
\ls & \| M^{-\e}(D_x^{\frac{s+1+\e+(\alpha+1)b'}{2}}u,D_x^{\frac{s+1+\e+(\alpha+1)b'}{2}}v)\|_{L^2_{xyt}},
\end{eqnarray*}
where we have used the assumption on the frequency sizes in this subcase. Observe that our choice of
$\beta$ implies $s+1+\e -\alpha \beta +b'=1+\e+2b' \ge 0$. Now the bilinear estimate
\eqref{Mepsbil} is applied to obtain the upper bound

$$\|D_x^{\frac{s+2+3\e+(\alpha+1)b'}{2}}u\|_{X_{0,b}}\|D_x^{\frac{s+2+3\e+(\alpha+1)b'}{2}}v\|_{X_{0,b}}\ls\|u\|_{X_{s,b}}\|v\|_{X_{s,b}},$$

where in the last step we have used \eqref{s}.

{\bf Subsubcase a.a.b:} $\lb \sigma_1 \rb \ge\lb \sigma \rb ,\lb \sigma_2 \rb $. Here the resonance relation
\eqref{rr} gives that the left hand side of \eqref{pf3} is bounded by

\begin{eqnarray*}
& \|D^{s+1+\e -\alpha \beta +b'}_x M^{-\e}(D_x^s\Lambda^bu, D_x^{\alpha(b'+\beta)-s}v)\|_{X_{0,-b}} \\
\ls & \|D^{-\frac12-\e}_x M^{-\e}(D_x^s\Lambda^bu, D_x^{\frac32 + 2 \e + (\alpha+1)b'}v)\|_{X_{0,-b}}.
\end{eqnarray*}

Now the dual version of estimate \eqref{Mepsbil}, that is

\begin{equation}\label{Mepsbil'}
\|M^{- \e} (u,v)\|_{X_{-\frac12 -,-\frac12 -}}\ls \|u\|_{L^2_{xyt}}\|v\|_{X_{\frac12 +,\frac12 +}}
\end{equation}

is applied, which gives, together with the assumption \eqref{s}, that the latter
is bounded by $c\|u\|_{X_{s,b}}\|v\|_{X_{s,b}}$. This completes the discussion of subcase a.a.
Concerning subcase a.b, where $|k| \gs |k_{1,2}|$, we solely remark that it can be
reduced to the estimation in subsubcase triple a.

{\bf Case b:} $\lb k\rb^{\alpha +1} \ge \lb \sigma \rb$. Here the additional weight
in the norm on the left of \eqref{nonlin1} is of size one, so our task is to show

\begin{equation}\label{pf4}
\|D^{s+1+\e}_x M^{-\e}(u,v)\|_{X_{0,b'}}\ls 
\|u\|_{X_{s,b;\beta}}\|v\|_{X_{s,b;\beta}}.
\end{equation}

{\bf Subcase b.a:} $\lb \sigma \rb \ge \lb \sigma_{1,2} \rb$. Here we may assume by symmetry
that $|k_1| \ge |k_2|$. We apply \eqref{rr} and \eqref{Mepsbil} to see that for $\delta \ge 0$
the left hand side of \eqref{pf4} is controlled by

\begin{eqnarray*}
 & \| M^{-\e}(D_x^{s+1+\e+\alpha b'+\delta}u,D_x^{b'-\delta}v)\|_{L^2_{xyt}}\\
 \ls & \|D_x^{\frac32 + 2 \e +\alpha b'+\delta}u\|_{X_{s,b}}\|D_x^{b'+ \frac12 + \e -\delta}v\|_{X_{0,b}}.
\end{eqnarray*}

The latter is bounded by $c\|u\|_{X_{s,b}}\|v\|_{X_{s,b}}$, provided $\frac32 + 2 \e +\alpha b'+\delta \le 0$
and $b'+ \frac12 + \e -\delta \le s$, which can be fulfilled by a proper choice of $\delta \ge 0$, since \eqref{s} holds.

{\bf Subcase b.b:} $\lb \sigma_1 \rb \ge\lb \sigma \rb ,\lb \sigma_2 \rb $. 

{\bf Subsubcase b.b.a:} $|k_1| \gs |k_2|$. With $\delta \ge 0$ as in subcase b.a the contribution here is bounded by

\begin{eqnarray*}
 & \|D_x^{-\frac12 -\e} M^{-\e}(D_x^{s+\frac32+2\e+\alpha b'+\delta}\Lambda^bu,D_x^{b'-\delta}v)\|_{X_{0,-b}}\\
 \ls & \|D_x^{\frac32 + 2 \e +\alpha b'+\delta}u\|_{X_{s,b}}\|D_x^{b'+ \frac12 + \e -\delta}v\|_{X_{0,b}} \ls \|u\|_{X_{s,b}}\|v\|_{X_{s,b}},
\end{eqnarray*}

where \eqref{rr} and the dual version \eqref{Mepsbil'} of Theorem \ref{Meps} were used again.
Finally we turn to the

{\bf Subsubcase triple b,} where $|k_1| \ll |k| \sim |k_2|$. Here the additional
weight in $\|u\|_{X_{s,b;\beta}}$ on the right of \eqref{pf4} behaves like

$$\left(\frac{|k|}{|k_1|} \right)^{\alpha \beta} \sim \left(\frac{|k_2|}{|k_1|} \right)^{\alpha \beta},$$

so that it is sufficient to show

\begin{equation}\label{pf5}
\|D^{s+1+\e}_x M^{-\e}(u,D_x^{-\alpha \beta}v)\|_{X_{0,b'}}\ls
\|u\|_{X_{s-\alpha \beta,b}}\|v\|_{X_{s,b}}.
\end{equation}

Now by \eqref{rr} the left hand side of \eqref{pf5} can be controlled by

\begin{eqnarray*}
 & \|D_x^{-\frac12 -\e} M^{-\e}(D_x^{b'}\Lambda^bu,D_x^{s + \frac32 + 2 \e +\alpha (b'-\beta)}v)\|_{X_{0,-b}}\\
 \ls & \|D_x^{b'}u\|_{X_{0,b}}\|D_x^{2 + 3\e + (\alpha +1)b'}v\|_{X_{0,b}}
\end{eqnarray*}

by \eqref{Mepsbil'}. Since $b'=s-\alpha \beta$ the first factor equals $\|u\|_{X_{s-\alpha \beta,b}}$,
while by \eqref{s} the second is dominated by $\|v\|_{X_{s,b}}$. This proves \eqref{pf5}.
\end{proof}

Finally we recall the definition of the Fourier restriction norm spaces from \cite{B93}.
For a time slab $I=(-\delta, \delta) \times \T^3$ they are given by

$$X^{\delta}_{s,\e , b; \beta} := \{u|_I : u \in X_{s,\e , b; \beta}\}$$

with norm

$$\|u\|_{X^{\delta}_{s,\e , b; \beta}}:= \inf \{\|\widetilde{u}\|_{X_{s,\e , b; \beta}}:\widetilde{u} \in X_{s,\e , b; \beta}, \widetilde{u}|I=u \}.$$

Now our well-posedness results read as follows.

\begin{theorem}[precise version of Theorem 3]
\label{finest}
Let $s \ge \frac12$ and $\e > 0$. Then for $u_0 \in H^{s}_xH^{\e }_y(\T^3)$ satisfying
\eqref{meanzero} there exist $\delta=\delta(\|u_0\|_{H^{s}_xH^{\e }_y})>0$ and a
unique solution $u \in X^{\delta}_{s,\e , \frac12; \frac12}$ of the Cauchy problem
\eqref{KPII} with $\alpha = 2$. This solution is persistent and the mapping $u_0 \mapsto u$,
$H^{s}_xH^{\e }_y(\T^3) \rightarrow X^{\delta_0}_{s,\e , \frac12; \frac12}$ is locally
Lipschitz for any $\delta_0 \in (0,\delta)$.
\end{theorem}

\begin{theorem}[precise version of Theorem 4]
\label{mothers}
Let $3 < \alpha \le 4$, $s \ge s' > \frac{3- \alpha}{2}$ and $\e \ge 0$. Then for
$u_0 \in H^{s}_xH^{\e }_y(\T^3)$ satisfying \eqref{meanzero} there exist $b>\frac12$, $\beta >0$,
$\delta=\delta(\|u_0\|_{H^{s'}_xH^{\e }_y})>0$ and a unique solution $u \in X^{\delta}_{s,\e , b; \beta} \subset C^0((-\delta, \delta),H^{s}_xH^{\e }_y(\T^3))$.
This solution depends continuously on the data, and extends globally in time, if $s\ge0$ and $\e =0$.
\end{theorem}

With the estimates from Lemma \ref{est0}
 and \ref{est1} at our disposal the proof
of these theorems is done by the contraction mapping principle, cf. \cite{B93}, \cite{GTV97},
\cite{KPV96a}, \cite{KPV96b}. The reader is also referred to section 1.3 of \cite{AG},
where the related arguments are gathered in a general local well-posedness theorem.

\end{document}